\documentclass[11pt,a4paper]{article}
\usepackage[a4paper,left=1.08in,right=1.08in,top=0.94in,bottom=0.76in]{geometry}
\usepackage{amsmath,amssymb,amsthm}
\usepackage{microtype}
\usepackage{needspace}
\usepackage{xcolor}
\usepackage{hyperref}
\usepackage{cite}
\hypersetup{
  pdftitle={Nonregular graphs of odd maximum degree with maximum spectral radius},
  pdfsubject={Revised manuscript incorporating the supplied PDF comments},
  colorlinks=true,
  linkcolor=red,
  citecolor=cyan!60!blue,
  urlcolor=blue
}
\makeatletter
\def\leftharpoonfill@{\arrowfill@\leftharpoonup\relbar\relbar}
\def\rightharpoonfill@{\arrowfill@\relbar\relbar\rightharpoonup}
\newcommand\rbjt{\mathpalette{\overarrow@\rightharpoonfill@}}
\newcommand\lbjt{\mathpalette{\overarrow@\leftharpoonfill@}}
\makeatother
\numberwithin{equation}{section}
\newtheorem{theorem}{Theorem}[section]
\newtheorem{lemma}[theorem]{Lemma}
\newtheorem{corollary}[theorem]{Corollary}
\newtheorem{conjecture}[theorem]{Conjecture}
\newtheorem{proposition}[theorem]{Proposition}
\theoremstyle{remark}

\newcommand{\cE}{\mathcal E}
\newcommand{\cJ}{\mathcal J}

\newcommand{\trans}{\mathsf T}
\newcommand{\doi}[1]{\href{https://doi.org/#1}{\nolinkurl{doi:#1}}}
\title{Nonregular graphs of odd maximum degree with maximum spectral radius}
\author{
Liangdong Fan$^1$, Liying Kang$^{2,3}$\footnote{Corresponding author. Email: lykang@shu.edu.cn}, Yaojun Chen$^1$\\
{\small $^1$School of Mathematics, Nanjing University, Nanjing 210093, P.R. CHINA}\\
{\small $^2$Department of Mathematics, Shanghai University, Shanghai 200444, P.R. CHINA}\\
{\small $^3$Newtouch Center for Mathematics of Shanghai University, Shanghai 200444, P.R. CHINA}
}
\date{}

\begin{document}
\pagestyle{plain}
\maketitle
\vspace{-1.1em}
\begin{abstract}

Let $\rho(n,d)$ denote the maximum adjacency spectral radius among all
connected nonregular graphs of order $n$ and maximum degree $d$.
A graph attaining this maximum is called an \emph{extremal graph}.

Liu [\emph{J. Combin. Theory Ser. B}, 2024] determined the extremal
graphs for $d=3,4$ and formulated two conjectures for general $d$.
For each fixed odd integer $d\ge3$, the conjectures assert that:
\begin{enumerate}
\setlength{\itemsep}{2pt}
\setlength{\parskip}{0pt}
\item $\displaystyle\lim_{n\to\infty}n^2\bigl(d-\rho(n,d)\bigr)
      =(d-1)\pi^2/4$.
\item For all sufficiently large $n$, the degree sequence of every
      extremal graph is $(d,\ldots,d,d-1)$ for odd $n$ and
      $(d,\ldots,d,1)$ for even $n$.
\end{enumerate}
We prove the first conjecture for every fixed odd $d\ge3$ and, more
precisely, obtain the asymptotic expansion
\[
 \rho(n,d)
 =d-\frac{(d-1)\pi^2}{4n^2}
   +\frac{(d-1)^2\pi^2}{4n^3}
   +O_d(n^{-4})
 \qquad(n\to\infty).
\]
We further prove the second conjecture for every fixed odd $d\ge3$.
\end{abstract}
\section{Introduction}\label{sec:intro}
All graphs are finite and undirected. Unless stated otherwise, they
are simple and connected. For a graph $G$, let $V(G)$ and $E(G)$
be its vertex and edge sets, respectively, let $d_G(v)$ be the degree of
$v\in V(G)$, and let $\Delta(G)$ be its maximum degree.
Let $A(G)$ denote the adjacency matrix of $G$.
The largest eigenvalue of $A(G)$, denoted by $\lambda_1(G)$,
is called the \emph{spectral radius} of $G$.
We write $K_m$ for the complete graph on $m$ vertices. 
It is well known that
$\lambda_1(G)\le\Delta(G)$,
with equality if and only if $G$ is regular. Hence
\[
\Delta(G)-\lambda_1(G)>0
\]
for every connected nonregular graph.
This naturally leads to the following extremal problem. 
\begin{quote} \textbf{Problem.} Among connected nonregular graphs of given order and maximum degree, how small can $\Delta(G)-\lambda_1(G)$ be, and which graphs attain this minimum? \end{quote}

For positive integers $n$ and $d$, let $\mathcal G(n,d)$ denote the
family of connected nonregular graphs of order $n$ and maximum degree
$d$. Whenever $\mathcal G(n,d)\ne\varnothing$, define
$\rho(n,d):=\max\{\lambda_1(G):G\in\mathcal G(n,d)\}$.
A graph $G\in\mathcal G(n,d)$ is called an \emph{extremal graph} if
$\lambda_1(G)=\rho(n,d)$. Thus the minimum in the problem above is
$d-\rho(n,d)$.

For a graph $G$ of order $n$, write
$\pi(G)=(d_1,\ldots,d_n)$, where $d_1\ge\cdots\ge d_n$, denote its degree sequence in nonincreasing order. For $G\in\mathcal G(n,d)$, the quantity
$\sum_{v\in V(G)}(d-d_G(v))=nd-2|E(G)|$ is called the
\emph{total degree deficiency} of $G$.

A subscript in the $O$- or $\Theta$-notation indicates the parameters
on which the implicit constant may depend.

\subsection{Background}
Stevanovi\'c~\cite{Stevanovic} first obtained a lower bound on
$d-\rho(n,d)$, and Zhang~\cite{Zhang} subsequently improved this bound.
Let $G\in\mathcal G(n,d)$ be an extremal graph, put $m:=|E(G)|$, and let
\(
D:=\operatorname{diam}(G)
  =\max_{u,v\in V(G)}\operatorname{dist}_G(u,v),
\)
where $\operatorname{dist}_G(u,v)$ is the length of a shortest
$u$--$v$ path. Cioab\u{a}, Gregory and Nikiforov~\cite{CGN} proved that
$d-\rho(n,d)>(nd-2m)/[n(D(nd-2m)+1)]$, and
Cioab\u{a}~\cite{Cioaba} later proved that
$d-\rho(n,d)>1/(nD)$.

Liu, Shen and Wang~\cite{LSW} proved that
$d-\rho(n,d)\ge(d+1)/[n(3n+2d-4)]$, and Liu, Li, Liu and You~\cite{LL,LLY} established the stronger bound
$d-\rho(n,d)>(d+1)/[n(3n+d-8)]$. Liu, Shen and Wang also established
the following asymptotic order for $d-\rho(n,d)$.
\begin{theorem}[Liu, Shen and Wang~\cite{LSW}]
\label{thm:lsw-order}
For every fixed integer $d\ge2$, one has
$d-\rho(n,d)=\Theta_d(n^{-2})$ as $n\to\infty$.
\end{theorem}

They further conjectured the precise leading asymptotic constant.

\begin{conjecture}[Liu--Shen--Wang~\cite{LSW}]
\label{conj:lsw}
For every fixed integer $d\ge2$,
\[
 \lim_{n\to\infty}
 \frac{n^2\bigl(d-\rho(n,d)\bigr)}{d-1}=\pi^2.
\]
\end{conjecture}

The case $d=2$ is immediate, since every graph in $\mathcal G(n,2)$ is
the path $P_n$ and $\lambda_1(P_n)=2\cos(\pi/(n+1))$.

Another conjecture of Liu and Li concerns the degree sequence of
extremal graphs.
\begin{conjecture}[Liu--Li~\cite{LL}]
\label{conj:liu-li}
Let $3\le d\le n-2$, and let $G\in\mathcal G(n,d)$ be extremal. Then
\[
 \pi(G)=
 \begin{cases}
 (d,\ldots,d,d-1),& nd\ \textnormal{is odd},\\
 (d,\ldots,d,d-2),& nd\ \textnormal{is even}.
 \end{cases}
\]
\end{conjecture}

Liu~\cite{Liu} disproved Conjecture~\ref{conj:lsw} for every $d\ge3$
and confirmed Conjecture~\ref{conj:liu-li} for $d\in\{3,4\}$ by
determining the corresponding extremal graphs. To state his asymptotic
bound, set
\[
 \kappa_d:=
 \begin{cases}
 \dfrac{(d-1)\pi^2}{4},& d\ \textnormal{is odd},\\[6pt]
 \dfrac{(d-2)\pi^2}{2},& d\ \textnormal{is even}.
 \end{cases}
\]
\begin{theorem}[Liu~\cite{Liu}]
\label{thm:liu}
For every fixed integer $d\ge3$,
\[
 \limsup_{n\to\infty}
 \Bigl[n^2\bigl(d-\rho(n,d)\bigr)\Bigr]\le\kappa_d.
\]
For $d\in\{3,4\}$, the sequence
$n^2\bigl(d-\rho(n,d)\bigr)$ converges to $\kappa_d$ as $n\to\infty$.
\end{theorem}
Theorem~\ref{thm:liu} leaves open whether
$n^2(d-\rho(n,d))$ converges and, if so, whether its limit equals
$\kappa_d$. Liu~\cite{Liu} conjectured that both are true, and
also proposed a conjecture on the degree sequence of extremal graphs.

\begin{conjecture}[Liu~\cite{Liu}]
\label{conj:asymptotic}
For every fixed integer $d\ge3$,
\[
 \lim_{n\to\infty}
 \Bigl[n^2\bigl(d-\rho(n,d)\bigr)\Bigr]=\kappa_d.
\]
\end{conjecture}
\begin{conjecture}[Liu~\cite{Liu}]
\label{conj:degrees}
For every fixed integer $d\ge3$ and all sufficiently large $n$, every
extremal graph $G\in\mathcal G(n,d)$ has degree sequence
\[
 \pi(G)=
 \begin{cases}
 (d,\ldots,d,d-1),&
 d\ \textnormal{and }n\ \textnormal{are odd},\\
 (d,\ldots,d,1),&
 d\ \textnormal{is odd and }n\ \textnormal{is even},\\
 (d,\ldots,d,d-2),&
 d\ \textnormal{is even}.
 \end{cases}
\]
\end{conjecture}

For odd $d\ge5$ and even $n$, Conjectures~\ref{conj:liu-li}
and~\ref{conj:degrees} give different predictions: the former asserts
that the degree sequence is $(d,\ldots,d,d-2)$, whereas the latter
predicts $(d,\ldots,d,1)$.

When $d$ is close to $n$, Huang, Liu and Yang~\cite{HLY}
characterized the extremal graphs for $d=n-2$ with $n\ge5$ and for
$d=n-3$ with $n\ge59$.
\subsection{Main results}
We prove Conjectures~\ref{conj:asymptotic} and~\ref{conj:degrees} for
every fixed odd integer $d\ge3$. In fact, we obtain a more precise
asymptotic expansion for $\rho(n,d)$, including the term of order
$n^{-3}$.

\begin{theorem}
\label{thm:asymptotic}
Let $d\ge3$ be a fixed odd integer. Then, for all sufficiently large
$n$,
\begin{equation}\label{eq:main}
 \rho(n,d)
 =
 d-\frac{(d-1)\pi^2}
 {4\left(n+\frac{d-1}{2}\right)^2}
 +O_d(n^{-4})
 =
 d-\frac{(d-1)\pi^2}{4n^2}
 +\frac{(d-1)^2\pi^2}{4n^3}
 +O_d(n^{-4}).
\end{equation}
\end{theorem}

Our second main result determines the degree sequence of extremal graphs
of sufficiently large order.
\begin{theorem}
\label{thm:degree-sequence}
Let $d\ge3$ be a fixed odd integer. Then there exists $n_0=n_0(d)$
such that every extremal graph $G\in\mathcal G(n,d)$ with $n\ge n_0$
has degree sequence
\begin{equation}\label{eq:degrees}
 \pi(G)=
 \begin{cases}
 (d,\ldots,d,d-1),& n\ \textnormal{is odd},\\
 (d,\ldots,d,1),& n\ \textnormal{is even}.
 \end{cases}
\end{equation}
\end{theorem}

The rest of the paper is organized as follows.
In Section~\ref{sec:cuts}, we reduce the spectral problem to a
one-dimensional variational problem on a path. This problem is analyzed
in Section~\ref{sec:path}, where we derive a general lower bound on
$d-\lambda_1(G)$ together with a strengthened version needed for the
degree-sequence argument.

In Section~\ref{sec:finish}, we construct a family of graphs $G_n$
which yields a matching upper bound on $d-\rho(n,d)$. Combining this
construction with the general lower bound proves
Theorem~\ref{thm:asymptotic}. We then show that any extremal graph
whose degree sequence is not listed in~\eqref{eq:degrees} satisfies the
stronger lower bound. For all sufficiently large $n$, this implies
$\lambda_1(G)<\lambda_1(G_n)$, contradicting the extremality of $G$ and proving
Theorem~\ref{thm:degree-sequence}.

Finally, Appendix~\ref{app:ordered-cut} proves the coordinate estimate
used in Lemma~\ref{lem:block-energy} and derives an explicit lower
bound on $d-\lambda_1(G)$ valid for every $G\in\mathcal G(n,d)$.
\section{Comparison with a path}\label{sec:cuts}
The aim of this section is to prove the lower bound
\eqref{eq:weighted-comparison} on $d-\lambda_1(G)$.
Fix an odd integer $d\ge3$, let $G\in\mathcal G(n,d)$, and put
$K:=d-1$, $c:=(d-2)/2$, and $\mu:=d-\lambda_1(G)>0$.
Let $x$ be the unit Perron vector of $G$.
Then $A(G)x=\lambda_1(G)x$, $x_v>0$ for every $v\in V(G)$, and
$\sum_{v\in V(G)}x_v^2=1$.
Choose a labelling $V(G)=\{v_1,\ldots,v_n\}$ satisfying
$x_{v_1}\le\cdots\le x_{v_n}$, and write $x_i:=x_{v_i}$ for
$1\le i\le n$ and $x_0:=0$.
Such eigenvector orderings are a standard tool in spectral extremal
problems; see, for example,
\cite{Fiedler,Guiduli,BGI,AGI,AGquartic,AGdiameter,Zhu}.

For integers $0\le a<b\le n$, write $J=[a,b):=\{a,\ldots,b-1\}$,
$\ell_J:=b-a$, and $t_J:=x_b-x_a$.
Given a partition $\cJ=\{[a_{j-1},a_j):1\le j\le r\}$ with
$0=a_0<\cdots<a_r=n$, define
\begin{equation}\label{eq:interpolation}
 g_{a+s}:=x_a+\frac{s}{\ell_J}t_J
 \qquad(J=[a,b)\in\cJ,\quad s=0,\ldots,\ell_J).
\end{equation}
For $1\le j<r$, both formulas at $a_j$ give $g_{a_j}=x_{a_j}$.
Thus $g$ is well defined, with $g_0=0$ and $g_n=x_n$.
Let $P$ be the path on $\{0,\ldots,n\}$ with edges $\{i,i+1\}$,
$0\le i<n$, and put $\cE_P(z):=\sum_{i=0}^{n-1}(z_{i+1}-z_i)^2$
for $z\in\mathbb R^{n+1}$.
We shall choose $\cJ$, a subfamily
$\cJ_{\mathrm{str}}\subseteq\cJ$, and a constant
$\gamma=\gamma(d)>0$ such that, if
$\mu=O_d(n^{-2})$, then, for all sufficiently large $n$,
\begin{equation}\label{eq:weighted-comparison}
 \mu\ge
 \frac{K\cE_P(g)+(\gamma/2)\sum_{J\in\cJ_{\mathrm{str}}}t_J^2}
      {\sum_{i=1}^n g_i^2+c g_n^2}
 -O_d(n^{-4}).
\end{equation}
A uniform formulation is given in Corollary~\ref{cor:weighted-norm}.
\subsection{Local comparisons and the choice of intervals}
\label{subsec:block-partition}
For $v\in V(G)$, put $\delta_G(v):=d-d_G(v)$. Expanding
$x^{\trans}(dI-A(G))x$ gives
\begin{equation}\label{eq:energy-identity}
 \mu=\sum_{uv\in E(G)}(x_u-x_v)^2
      +\sum_{v\in V(G)}\delta_G(v)x_v^2.
\end{equation}
To express $\sum_v\delta_G(v)x_v^2$ as edge energy, we construct the
graph $G^+$ from $G$ by adjoining a new vertex $v_0$ and adding
$\delta_G(v)$ edges between $v_0$ and each $v\in V(G)$.
We extend $x$ to $V(G^+)$ by setting $x_{v_0}=x_0=0$.
Distinct edges with the same pair of endpoints are called
\emph{parallel edges}. Thus $G^+$ is a loopless multigraph.
Throughout this section, parallel edges are retained as distinct edges
and are counted separately in degrees, cuts, and energy sums.
For any such graph $H$ and any function $z:V(H)\to\mathbb R$, define
$\cE_H(z):=\sum_{uv\in E(H)}(z_u-z_v)^2$,
where edges are counted with multiplicity. We call $\cE_H(z)$ the
\emph{edge energy} of $z$ on $H$.
Equation~\eqref{eq:energy-identity} now gives $\cE_{G^+}(x)=\mu$.
Since $\sum_v\delta_G(v)\ge1$ and $G$ is connected, $G^+$ is connected.
Moreover, $G^+[V(G)]=G$ and $d_{G^+}(v)=d$ for every $v\in V(G)$.

For $S\subseteq V(G^+)$, let $\partial S$ be the set of edges with
exactly one endpoint in $S$. Put $S_i:=\{v_0,\ldots,v_i\}$ and
$q_i:=|\partial S_i|$ for $0\le i\le n$, and put
$y_i:=x_{i+1}-x_i\ge0$ for $0\le i<n$.
Thus $q_i\ge1$ for $i<n$ and $q_n=0$.
An edge $v_pv_q\in E(G^+)$ with $p<q$ belongs to $\partial S_i$ exactly when
$p\le i<q$. Expanding $(x_q-x_p)^2=(\sum_{i=p}^{q-1}y_i)^2$
and summing over $E(G^+)$ gives
\begin{equation}\label{eq:cut-energy}
 \cE_{G^+}(x)
 =\sum_{i,j=0}^{n-1}|\partial S_i\cap\partial S_j|\,y_i y_j.
\end{equation}
Since $y_i\ge0$, the diagonal terms give
$\mu\ge\sum_{i=0}^{n-1}q_i y_i^2$. If $q_i\ge K$, the term $q_i y_i^2$ already bounds $Ky_i^2$; if
$q_i<K$, we instead choose an interval $J$ with $i\in J$ and seek a
lower bound of the form $K t_J^2/\ell_J$.

Fix an interval $J=[a,b)\subseteq[0,n)$.
We construct the graph $H_J$ from $G^+$ by identifying
$v_0,\ldots,v_a$ into a single vertex $a$ and
$v_b,\ldots,v_n$ into a single vertex $b$, while relabelling
$v_i$ as $i$ for $a<i<b$. We delete the loops created by the
identifying and retain all parallel edges. Since $G^+$ is connected,
so is $H_J$.

The edges incident with $a$ in $H_J$ correspond exactly to
$\partial S_a$, and for each $a<i<b$, the degree of $i$ in
$H_J$ is the same as the degree of $v_i$ in $G^+$. Hence
$d_{H_J}(a)=q_a$ and $d_{H_J}(i)=d$ for $a<i<b$.
Moreover,
$H_J-\{a,b\}\cong G[\{v_i:a<i<b\}]$,
and hence $H_J-\{a,b\}$ is simple, but need not be connected.

If an edge $v_pv_q\in E(G^+)$ with $p<q$ is not turned into a loop
under the above identification, then its contribution to the energy of
$H_J$ is
\[
 \bigl(x_{\min\{b,q\}}-x_{\max\{a,p\}}\bigr)^2
 =
 \left(\sum_{k\in J\cap[p,q)}y_k\right)^2.
\]
For every edge that is turned into a loop, the sum on the right-hand
side is zero. For brevity, write $\cE_J(x):=\cE_{H_J}(x)$.
Expanding and summing
therefore gives
\begin{equation}\label{eq:block-graph-energy}
 \cE_J(x)
 =\sum_{ij\in E(H_J)}(x_i-x_j)^2
 =\sum_{\substack{v_pv_q\in E(G^+)\\p<q}}
   \left(\sum_{k\in J\cap[p,q)}y_k\right)^2.
\end{equation}
Consequently, with the boundary values $x_a$ and $x_b$ fixed,
$\cE_J(x)$ is bounded below by the minimum energy over all choices of
the interior values $x_{a+1},\ldots,x_{b-1}$ on $H_J$.

For a connected graph $H$ with two distinct specified vertices $L,R$,
allowing parallel edges but no loops, define
\[
 C(H):=\min_{\substack{w:V(H)\to\mathbb R\\w_L=1,\ w_R=0}}
              \sum_{uv\in E(H)}(w_u-w_v)^2.
\]
\Needspace{9\baselineskip}

The specified vertices $L$ and $R$ are understood from the context.
The following elementary lemma allows arbitrary boundary values and
also gives bounds on the coordinates of the minimizer that will be
used in the norm comparison.
\begin{lemma}\label{lem:prescribed-values}
For $A,B\in\mathbb R$,
\[
 \min_{\substack{w:V(H)\to\mathbb R\\w_L=A,\ w_R=B}}
       \cE_H(w)=C(H)(B-A)^2.
\]
The minimum is attained at a unique $w$, with
$\min\{A,B\}\le w_v\le\max\{A,B\}$ for every $v\in V(H)$.
\end{lemma}
\begin{proof}
Let $z:V(H)\to\mathbb R$ satisfy $z_L=1$ and $z_R=0$.
Replacing each $z_v$ by $\min\{1,\max\{0,z_v\}\}$ preserves these
values and cannot increase any $|z_p-z_q|$. Hence, in computing $C(H)$, it suffices to restrict to functions
satisfying $0\le z_v\le1$ for all $v\in V(H)$.
The function $\cE_H$ is continuous in these $|V(H)|-2$ coordinates,
so it attains its minimum on $[0,1]^{|V(H)|-2}$.
Let $h$ attain this minimum; then $0\le h_v\le1$ for all $v$.
If $z$ is another minimizer, then
$C(H)\le\cE_H\!\left((h+z)/2\right)=C(H)-(1/4)\cE_H(h-z)$.
Hence $\cE_H(h-z)=0$. Since $H$ is connected, $h-z$ is constant on
$V(H)$. As $h_L=z_L=1$, this constant is zero, and therefore $h=z$.

For $A\ne B$, the affine change of variables
\(
 w=B+(A-B)z
\)
maps the boundary conditions $z_L=1$, $z_R=0$ to
$w_L=A$, $w_R=B$, and
$\cE_H(w)=(B-A)^2\cE_H(z)$. The unique minimizer is therefore
$w=B+(A-B)h$, which has the stated coordinate bounds.
For $A=B$, the function $w_v=A$ has energy zero. Conversely,
$\cE_H(w)=0$ gives $w_p=w_q$ on each edge $pq$; along a path from
$L$ to $v$, this yields $w_v=w_L=A$.
\end{proof}

For $H_J$, take $a,b$ as the specified vertices and put
$C_J:=C(H_J)$. By Lemma~\ref{lem:prescribed-values}, if the values
at $a$ and $b$ are fixed to be $x_a$ and $x_b$, respectively, then
the minimum energy on $H_J$ is
$C_J(x_b-x_a)^2=C_Jt_J^2$.
On the other hand, by~\eqref{eq:block-graph-energy}, the values
$x_a,\ldots,x_b$ have energy $\cE_J(x)$ on $H_J$. Therefore,
\[
 \cE_J(x)
 =\sum_{ij\in E(H_J)}(x_i-x_j)^2
 \ge
 \min_{\substack{w:V(H_J)\to\mathbb R\\
                  w_a=x_a,\ w_b=x_b}}
       \sum_{ij\in E(H_J)}(w_i-w_j)^2
 =C_Jt_J^2.
\]

Thus $\cE_J(x)$ is the energy of the original coordinates on $H_J$,
whereas $C_Jt_J^2$ is the minimum energy subject to the boundary
values $x_a$ and $x_b$.
It remains to identify intervals for which
\(C_J\ge K/\ell_J\).
If $J=\{a\}$, then $H_J$ consists of $q_a$ parallel edges between
$a$ and $a+1$, and hence $C_J=q_a$. Thus any index with $q_a\ge K$
can be handled by a singleton interval.

It remains to consider the case $q_a<K$. The following lemma gives the
required bounds when $\ell_J=d$ or $d+1$, equivalently, when
$H_J-\{a,b\}$ has $d-1$ or $d$ vertices.
\begin{lemma}\label{lem:block-energy}
Let $J=[a,b)$ be an integer interval with $0\le a<b\le n$.
If $\ell_J=d+1$ and $1\le q_a\le d-2$, then $C_J\ge K/(d+1)$.
If $\ell_J=d$ and $2\le q_a\le d-2$, then $C_J>K/d$.
In the first case, $C_J=K/(d+1)$ if and only if $q_a=1$ and,
for some $i_*$ with $a<i_*<b$,
\[
 E(H_J)=\{ij:a<i<j<b\}\cup\{ai_*\}
                    \cup\{ib:a<i<b,\ i\ne i_*\},
\]
with each listed edge occurring once.
\end{lemma}
\begin{proof}
If $ab\in E(H_J)$, then $C_J\ge1>K/d>K/(d+1)$.
Assume $ab\notin E(H_J)$, and let $w$ minimize
$\cE_{H_J}(w)=\sum_{ij\in E(H_J)}(w_i-w_j)^2$ with $w_a=1$, $w_b=0$.
Lemma~\ref{lem:prescribed-values} gives $0\le w_i\le1$ for $a\le i\le b$.
For $a<i<b$, let $p_i$ count the edges joining $a$ to $i$;
then $\sum_{a<i<b}p_i=q_a$.
Since $d_{H_J}(i)=d$ for $a<i<b$, the optimality equations for $w$ give
$dw_i-p_i=\sum_{\substack{a<j<b\\ij\in E(H_J)}}w_j$ for $a<i<b$.
Consequently,
\[
 \sum_{a<i<b}(d w_i^2-p_iw_i)
 =\sum_{a<i<b}\sum_{\substack{a<j<b\\ij\in E(H_J)}}w_iw_j
 =2\sum_{\substack{a<i<j<b\\ij\in E(H_J)}}w_iw_j.
\]
Using $w_a=1$, $w_b=0$, and $ab\notin E(H_J)$, we expand
$\cE_{H_J}(w)$ as
\[
 \begin{aligned}
 C_J
 &=q_a+\sum_{a<i<b}(d w_i^2-2p_iw_i)
       -2\sum_{\substack{a<i<j<b\\ij\in E(H_J)}}w_iw_j\\
 &=\sum_{a<i<b}p_i(1-w_i)
 \ge q_a\left(1-\max_{a<i<b}w_i\right).
 \end{aligned}
\]
The graph $H_J-\{a,b\}$ is simple and has $\ell_J-1$ vertices.
Applying Lemma~\ref{lem:local-energy} in Appendix~\ref{app:local-energy}
to $\max_{a<i<b}w_i$, with $H=H_J$, $(L,R)=(a,b)$,
$m=\ell_J-1$, and $k=q_a$, gives
\[
 C_J\ge
 \begin{cases}
 \dfrac{q_a(d-q_a)}{d+q_a},&\ell_J=d+1,\\[5pt]
 \dfrac{2q_a(d-q_a)}{2d+q_a},&\ell_J=d.
 \end{cases}
\]
If $\ell_J=d+1$ and $q_a=1$, this gives $C_J\ge K/(d+1)$.
For $2\le q_a\le d-2$, we have $d\ge5$ and
$q_a(d-q_a)=2(d-2)+(q_a-2)(d-2-q_a)\ge2(d-2)$.
Thus $C_J\ge(d-2)/(d-1)>K/(d+1)$ when $\ell_J=d+1$, and
$C_J\ge4(d-2)/(3d-2)>K/d$ when $\ell_J=d$.

Suppose now that $\ell_J=d+1$ and $C_J=K/(d+1)$.
By the strict inequality established above for $q_a\ge2$, we must have
$q_a=1$.
Let $i_*$ be the unique neighbour of $a$. Then $p_{i_*}=1$ and
$p_i=0$ for $a<i<b$, $i\ne i_*$. Hence $C_J=1-w_{i_*}$, so
$w_{i_*}=2/(d+1)$.
Put $S:=\sum_{a<i<b}w_i$. By the optimality equations and the simplicity of
$H_J-\{a,b\}$,
\[
 dw_i
 =p_i+\sum_{\substack{a<j<b\\ij\in E(H_J)}}w_j
 \le p_i+S-w_i,
\]
and hence
$(d+1)w_i\le p_i+S$ for $a<i<b$.
At $i_*$ this yields $2\le1+S$; summing over the other $d-1$
indices yields $(d+1)(S-w_{i_*})\le(d-1)S$.
Hence $S=1$. Since $w_i\le1/(d+1)$ for $a<i<b$, $i\ne i_*$, and
$\sum_{a<i<b,\,i\ne i_*}w_i=(d-1)/(d+1)$, we obtain
$w_i=1/(d+1)$ for $a<i<b$, $i\ne i_*$.
For $a<i<b$, $i\ne i_*$, we therefore have
\[
 0=S-(d+1)w_i
   =\sum_{\substack{a<j<b,\ j\ne i\\ij\notin E(H_J)}}w_j.
\]
Since $w_j\ge1/(d+1)>0$ for $a<j<b$, every vertex $i$ with $a<i<b$ and
$i\ne i_*$ is adjacent to every other vertex of $H_J-\{a,b\}$.
Thus $H_J-\{a,b\}$ is complete.
Since $d_{H_J}(i)=d$ for every $a<i<b$ and $q_a=1$, it follows that
$i_*b\notin E(H_J)$, while for every $a<i<b$ with $i\ne i_*$ there is
exactly one edge between $i$ and $b$.
Together with $ab\notin E(H_J)$, this gives precisely the edge set
stated in the lemma.

Conversely, suppose that $H_J$ has the stated edge set.
Every permutation of the vertices $a<i<b$, $i\ne i_*$, extends to an automorphism of $H_J$ fixing $a$, $b$, and $i_*$. Lemma~\ref{lem:prescribed-values} therefore gives
$w_i=w_j$ for $a<i,j<b$ with $i,j\ne i_*$.
For $a<i<b$, $i\ne i_*$, the equations
$dw_{i_*}=1+(d-1)w_i$ and $2w_i=w_{i_*}$ give
$w_{i_*}=2/(d+1)$ and $w_i=1/(d+1)$.
Consequently $C_J=1-w_{i_*}=K/(d+1)$.
\end{proof}

The conclusion $C(H)\ge K/d$ need not hold when
$|V(H)\setminus\{L,R\}|=d-1$ and $d_H(L)=1$, even if every
interior vertex has degree $d$.
Indeed, let $F=K_{d-1}$, choose $v_*\in V(F)$, and form $H$ by
adjoining vertices $L,R$, the edges $Lv_*,v_*R$, and two edges
$vR$ for each $v\in V(F)\setminus\{v_*\}$.
Each $v\in V(F)$ satisfies $d_H(v)=d$, and $d_H(L)=1$.
By symmetry and the optimality equations, the minimizing function
with $w_L=1$ and $w_R=0$ satisfies
$w_{v_*}=3/(2d+2)$ and $w_v=1/(2d+2)$ for $v\ne v_*,L,R$.
Substitution in $\cE_H(w)$ gives $C(H)=(2d-1)/(2d+2)<K/d$.
This explains the different choices of interval length below: when
$q_a=1$ we use an interval of length $d+1$, whereas for
$2\le q_a<K$ the length-$d$ estimate applies.

We next show that these intervals can form a partition of $[0,n)$.
For $0\le i<j\le n$, put $U:=S_j\setminus S_i$ and $t:=j-i$.
Each vertex of $U\subseteq V(G)$ has degree $d$ in $G^+$.
Since $G^+[U]$ is simple and
$\partial U\subseteq\partial S_i\cup\partial S_j$,
\begin{equation}\label{eq:cut-spacing}
 q_i+q_j\ge|\partial U|
 =dt-2|E(G^+[U])|\ge t(d-t+1).
\end{equation}
If $q_i,q_j<K$ and $2\le t\le d-1$, this contradicts
$q_i+q_j\le2(d-2)<2(d-1)\le t(d-t+1)$; the last inequality follows
from $(t-2)(d-1-t)\ge0$.
For $1\le t\le d$, we also have
$t(d-t+1)\ge d$, since $t(d-t+1)-d=(t-1)(d-t)\ge0$.
Thus $q_i=1$, $q_j<K$, and $1\le|i-j|\le d$ would give
$d\le q_i+q_j\le d-1$, a contradiction.
Thus distinct $i,j\in[0,n)$ with $q_i,q_j<K$ satisfy
\[
 |i-j|=1\ \text{or}\ |i-j|\ge d,
 \qquad q_i=1\ \Longrightarrow\ |i-j|\ge d+1.
\]
Taking $j=n$ in~\eqref{eq:cut-spacing} also gives $q_i\ge d$
for $1\le n-i\le d$, because $q_n=0$.
Thus every index with $q_i<K$ satisfies $i\le n-d-1$.
These separation and endpoint bounds yield the following partition.
\begin{lemma}\label{lem:block-partition}
There is a partition $\cJ$ of $[0,n)$ into integer intervals such that
every $J=[a,b)\in\cJ$ satisfies one of
\[
 \begin{array}{ll}
 \ell_J=1,&q_a\ge K,\\
 \ell_J=d,&2\le q_a\le d-2,\\
 \ell_J=d+1,&1\le q_a\le d-2.
 \end{array}
\]
In particular, $\cE_J(x)\ge Kt_J^2/\ell_J$ for every $J\in\cJ$.
\end{lemma}
\begin{proof}
Let $B\subseteq[0,n)$ be a maximal nonempty integer interval with
$q_i<K$ for every $i\in B$.
The excluded distance $2$ gives $|B|\le2$.
For $a:=\min B$, set
\begin{equation}\label{eq:block-rule}
 J(B):=
 \begin{cases}
 [a,a+d+1),&B=\{a\},\ q_a=1,\\
 [a,a+d),&B=\{a\},\ 2\le q_a\le d-2,\\
 [a,a+d+1),&B=\{a,a+1\}.
 \end{cases}
\end{equation}
Let $B$ and $B'$ be two consecutive maximal intervals of this type,
with first indices $a<a'$. The
first case of~\eqref{eq:block-rule} gives $a'-a\ge d+1$.
In the second case, $a'-a\ge2$, and hence $a'-a\ge d$ by the
distance restriction above.
In the third case, $a'-(a+1)\ge2$, and hence
$a'-(a+1)\ge d$ by the same restriction, so $a'-a\ge d+1$.
Thus the intervals $J(B)$ are disjoint.
Taking $j=n$ in~\eqref{eq:cut-spacing} shows that every index
$i$ with $q_i<K$ satisfies $i\le n-d-1$. Hence
$a\le n-d-1$.
As $|J(B)|\le d+1$, we also have $J(B)\subseteq[0,n)$.

Add $\{i\}$ for every $i\in[0,n)\setminus\bigcup_BJ(B)$.
Each added singleton satisfies $q_i\ge K$, because every index with
$q_i<K$ belongs to some $B\subseteq J(B)$.
This gives the required partition; its energy bound follows from
Lemma~\ref{lem:block-energy} and $C_{\{a\}}=q_a$.
\end{proof}

Fix the partition constructed in the proof of
Lemma~\ref{lem:block-partition}, write
$\cJ=\{[a_{j-1},a_j):1\le j\le r\}$ with $0=a_0<\cdots<a_r=n$,
and define $g$ by~\eqref{eq:interpolation}.

By~\eqref{eq:cut-energy} and~\eqref{eq:block-graph-energy},
$\cE_J(x)=\sum_{i,j\in J}|\partial S_i\cap\partial S_j|\,y_i y_j$ for $J\in\cJ$.
Since the blocks in $\cJ$ are pairwise disjoint and $y_i\ge0$,
\begin{equation}\label{eq:sum-blocks}
 \sum_{J\in\cJ}\cE_J(x)
 \le \cE_{G^+}(x)
 =\mu.
\end{equation}

Call each $J\in\cJ$ a \emph{block}.
To keep track of the blocks for which the preceding comparison is
strict, define
$\cJ_{\mathrm{str}}:=\left\{J\in\cJ:C_J>K/\ell_J\right\}$,
and call the members of $\cJ_{\mathrm{str}}$ \emph{strict blocks}.
Set $\varepsilon_J=1$ if $J\in\cJ_{\mathrm{str}}$ and
$\varepsilon_J=0$ otherwise.

For fixed $d$, the positive differences
$C_J-K/\ell_J$ are uniformly bounded away from zero.
Indeed, if $\ell_J>1$, then every edge of $H_J$ either has an
endpoint in $\{a+1,\ldots,b-1\}$ or joins $a$ to $b$. Hence
\[
 |V(H_J)|=\ell_J+1\le d+2,
 \qquad
 |E(H_J)|
 \le d(\ell_J-1)+q_a
 \le d^2+d-2.
\]
Thus, for fixed $d$, there are only finitely many possibilities for
the multigraph $H_J$ together with its specified vertices $a,b$, up
to isomorphism. Consequently, there exists a constant
$\gamma=\gamma(d)\in(0,1]$ such that
$C_J-K/\ell_J\ge\gamma$
whenever $\ell_J>1$ and $C_J>K/\ell_J$.
For a strict singleton $J=\{a\}$, we have
$C_J-K=q_a-K\ge1$. Therefore, for every $J\in\cJ$,
\begin{equation}\label{eq:block-energy}
 \cE_J(x)\ge C_Jt_J^2,\qquad
 C_J\ge\frac K{\ell_J}+\gamma\varepsilon_J.
\end{equation}

Now $g_{i+1}-g_i=t_J/\ell_J$ for $i\in J$, and hence
\begin{equation}\label{eq:interpolated-path-energy}
 \cE_P(g)=\sum_{J\in\cJ}\frac{t_J^2}{\ell_J}.
\end{equation}
Summing the first inequality in~\eqref{eq:block-energy} over
$J\in\cJ$, and then using~\eqref{eq:sum-blocks} and
\eqref{eq:interpolated-path-energy}, we obtain
\begin{equation}
    \label{eq:coarse-energy}
 K\cE_P(g)
 +\gamma\sum_{J\in\cJ_{\mathrm{str}}}t_J^2
 \le\mu.
\end{equation}
\subsection{The norm correction}
For $z=(z_0,\ldots,z_n)$ with $z_0=0$, write
$\|z\|_2^2:=\sum_{i=1}^n z_i^2$.
Recall that $g_{a_j}=x_{a_j}$ and, for $J=[a,b)\in\cJ$ and $a<i<b$,
$x_a\le x_i,g_i\le x_b$. Thus $|x_i-g_i|\le t_J$, and
\eqref{eq:block-energy}, $\ell_J\le d+1$, and~\eqref{eq:sum-blocks} give
\begin{equation}\label{eq:coarse-norm}
 \|x-g\|_2^2
 \le\sum_{J\in\cJ}(\ell_J-1)t_J^2
 \le\frac{d(d+1)}K\sum_{J\in\cJ}\cE_J(x)
 \le\frac{d(d+1)}K\,\mu.
\end{equation}
For $\mu=O_d(n^{-2})$, this gives $\|g\|_2\ge1-O_d(n^{-1})$.
To compare the squared norms more precisely, put
$D_J:=\sum_{a<i<b}(x_i^2-g_i^2)$ for $J=[a,b)\in\cJ$.
Since $x_{a_j}=g_{a_j}$ for every $j$ and $\|x\|_2=1$,
\(
 \sum_{J\in\cJ}D_J=1-\|g\|_2^2.
\)
Together with~\eqref{eq:sum-blocks}, this gives
\begin{equation}\label{eq:sum-norm}
 \sum_{J\in\cJ}\bigl(\cE_J(x)-\mu D_J\bigr)
 \le\mu\|g\|_2^2.
\end{equation}
We therefore seek a lower bound on $\cE_J(x)-\mu D_J$.
We shall extract a term
\(
 -\mu c(x_b^2-x_a^2)
\)
from each block. Since the blocks form a partition of $[0,n)$ and
$x_0=0$, these terms telescope to $-\mu c x_n^2$. After moving this
term to the right-hand side, it becomes $\mu c g_n^2$, since
$g_n=x_n$.
\begin{proposition}\label{prop:refined}
There is a constant $C_d>0$, depending only on $d$, such that, if
$\mu\le1/4$,
\begin{equation}\label{eq:refined}
 K\cE_P(g)+\frac\gamma2\sum_{J\in\cJ_{\mathrm{str}}}t_J^2
 \le\mu\left(\sum_{i=1}^n g_i^2+c g_n^2\right)+C_d\mu^2.
\end{equation}
\end{proposition}
\begin{proof}
We first prove, for every $J=[a,b)\in\cJ$, the estimate
\begin{equation}\label{eq:local-refined}
 \cE_J(x)-\mu D_J
 \ge \frac K{\ell_J}t_J^2+\frac\gamma2\varepsilon_Jt_J^2
        -\mu c(x_b^2-x_a^2)
 -B_d\mu t_J^2-B_d\mu^2(x_a^2+x_b^2),
\end{equation}
where $B_d>0$ depends only on $d$. We will show below that, after summing over $J\in\cJ$, the last two
terms contribute only $O_d(\mu^2)$.
For $\ell_J=1$, we have $D_J=0$ and
$\cE_J(x)=C_Jt_J^2\ge(K+\gamma\varepsilon_J)t_J^2$.
Since $x_b^2-x_a^2\ge0$, this implies~\eqref{eq:local-refined}
for any $B_d>0$.
Now fix $J=[a,b)\in\cJ$ with $\ell:=\ell_J>1$ and put
$A:=x_a$, $B:=x_b$, and $t:=B-A$, so $0\le A\le B$ and $t\ge0$.

Let $h$ attain the minimum on $H_J$ with $h_a=0$ and $h_b=1$.
By Lemma~\ref{lem:prescribed-values}, it is unique,
$\cE_{H_J}(h)=C_J$, and $0\le h_i\le1$.
For $a\le i\le b$, put $w_i:=A+t h_i$ and $u_i:=x_i-w_i$.
Then $w_a=A$, $w_b=B$, $u_a=u_b=0$, and $w$ attains the minimum
$C_Jt^2$ for these specified values.
For every $\tau\in\mathbb R$, $w+\tau u$ still has values $A,B$
at $a,b$, so $\cE_{H_J}(w+\tau u)\ge\cE_{H_J}(w)$.
Differentiating at $\tau=0$ gives
\[
 0=\frac12\left.\frac{d}{d\tau}\cE_{H_J}(w+\tau u)\right|_{\tau=0}
 =\sum_{ij\in E(H_J)}(w_i-w_j)(u_i-u_j).
\]

Since $x=w+u$ on $V(H_J)$, expansion in
\eqref{eq:block-graph-energy} and the orthogonality relation above yields
\begin{equation}\label{eq:orthogonal-energy}
 \cE_J(x)=C_Jt^2+\sum_{ij\in E(H_J)}(u_i-u_j)^2.
\end{equation}
The $\ell-1$ interior vertices induce a simple graph, and hence each
of them has at most $\ell-2$ neighbours among the interior vertices.
Since $d_{H_J}(i)=d$ for every $a<i<b$, at least
$d-\ell+2$ edges incident with $i$ join $i$ to $a$ or $b$.
Using $u_a=u_b=0$, we obtain
\begin{equation}\label{eq:block-error-energy}
 \sum_{ij\in E(H_J)}(u_i-u_j)^2
 \ge(d-\ell+2)\sum_{a<i<b}u_i^2
 \ge\sum_{a<i<b}u_i^2.
\end{equation}

Set $\alpha:=\sum_{a<i<b}(h_i-(i-a)/\ell)$.
Since $x_i=A+t h_i+u_i$ and $g_i=A+(i-a)t/\ell$, we have
\[
 D_J=2At\alpha
       +t^2\sum_{a<i<b}\left(h_i^2-\frac{(i-a)^2}{\ell^2}\right)
       +2\sum_{a<i<b}w_i u_i+\sum_{a<i<b}u_i^2.
\]
To compare $\mu D_J$ with $\mu c(B^2-A^2)$, the main term that
requires attention is
$2\mu At(\alpha-c)$. We will show that this term vanishes when
$\varepsilon_J=0$; when $\varepsilon_J=1$, we use half of
$\gamma t^2$ in~\eqref{eq:block-energy} to bound it.
\Needspace{9\baselineskip}

If $\varepsilon_J=0$, then $J\notin\cJ_{\mathrm{str}}$, and hence
$C_J\le K/\ell$ by the definition of $\cJ_{\mathrm{str}}$.
On the other hand,~\eqref{eq:block-energy} gives $C_J\ge K/\ell$.
Therefore, \(
 C_J= K/\ell
\).
As $\ell>1$, Lemma~\ref{lem:block-partition} gives
$\ell\in\{d,d+1\}$; the strict bound in Lemma~\ref{lem:block-energy}
excludes $\ell=d$. Hence $\ell=d+1$, and we choose $i_*$ from the
equality case of Lemma~\ref{lem:block-energy}.
The map $z\mapsto1-z$ preserves $\cE_{H_J}(z)$ and interchanges
the specified values $1,0$ with $0,1$. By the uniqueness in
Lemma~\ref{lem:prescribed-values}, the equality case computation in the proof of
Lemma~\ref{lem:block-energy} therefore gives
$h_{i_*}=(d-1)/(d+1)$ and $h_i=d/(d+1)$ for
$a<i<b$, $i\ne i_*$.
Thus
\[
 \alpha=\frac{d-1}{d+1}+(d-1)\frac d{d+1}
            -\sum_{s=1}^d\frac{s}{d+1}
         =d-1-\frac d2=c.
\]
If $\varepsilon_J=1$, the bounds $0\le h_i,(i-a)/\ell\le1$ give
$|\alpha|\le\ell-1\le d$. Thus
$|\alpha-c|\le(d+c)\varepsilon_J$ in both cases.
Also, $\ell-1\le d$ and $A\le w_i\le B$ gives $\sum_{a<i<b}w_i^2\le dB^2$.
Substituting
\(
 B^2-A^2=2At+t^2
\)
into the preceding expression for $D_J$, and using
$\sum_{a<i<b}\left(h_i^2-(i-a)^2/\ell^2\right)-c\le d+c$,
we obtain
\begin{equation}\label{eq:norm-splitting}
 D_J\le c(B^2-A^2)+2At(\alpha-c)+(d+c)t^2
 +2\sum_{a<i<b}w_i u_i+\sum_{a<i<b}u_i^2.
\end{equation}
Applying Young's inequality $2\xi\zeta\le\eta\xi^2+\eta^{-1}\zeta^2$ for $\eta>0$
gives
\begin{align*}
 2\mu|At(\alpha-c)|
 &\le\frac\gamma2\varepsilon_Jt^2
       +\frac{2(d+c)^2}{\gamma}\mu^2A^2,\\
 \mu\left(2\sum_{a<i<b}w_i u_i+\sum_{a<i<b}u_i^2\right)
 &\le\left(\frac14+\mu\right)\sum_{a<i<b}u_i^2
       +4\mu^2\sum_{a<i<b}w_i^2\\
 &\le\frac12\sum_{a<i<b}u_i^2+4d\mu^2B^2,
\end{align*}
where in the last inequality we used $\mu\le1/4$.

Take $B_d:=\max\{d+c,2(d+c)^2/\gamma,4d\}$.
Substitution in~\eqref{eq:norm-splitting} gives
\[
 \mu D_J\le \mu c(B^2-A^2)+\frac\gamma2\varepsilon_Jt^2
              +\frac12\sum_{a<i<b}u_i^2
 +B_d\mu t^2+B_d\mu^2(A^2+B^2).
\]
On the other hand,~\eqref{eq:block-energy},
\eqref{eq:orthogonal-energy}, and~\eqref{eq:block-error-energy} give
$\cE_J(x)\ge(K/\ell+\gamma\varepsilon_J)t^2+\sum_{a<i<b}u_i^2$.
Subtracting the bound on $\mu D_J$ and discarding
$(1/2)\sum_{a<i<b}u_i^2\ge0$ proves~\eqref{eq:local-refined}.
The argument remains valid when $A=0$ or $t=0$, since no division by
either quantity was used.
\Needspace{8\baselineskip}

We now sum~\eqref{eq:local-refined}. Since $x_0=0$,
\begin{equation}\label{eq:sum-endpoint-corrections}
 \sum_{J=[a,b)\in\cJ}(x_b^2-x_a^2)
 =\sum_{j=1}^r(x_{a_j}^2-x_{a_{j-1}}^2)=x_n^2.
\end{equation}
Also, $\|x\|_2=1$, $\ell_J\le d+1$, \eqref{eq:sum-blocks}, and~\eqref{eq:block-energy}
imply
\[
 \sum_{J=[a,b)\in\cJ}(x_a^2+x_b^2)
 =2\sum_{j=1}^{r-1}x_{a_j}^2+x_n^2\le2,
 \qquad \sum_{J\in\cJ}t_J^2
 \le\frac{d+1}K\sum_{J\in\cJ}\cE_J(x)
 \le\frac{d+1}K\,\mu.
\]
The remainders in~\eqref{eq:local-refined} therefore sum to at most
\[
 B_d\mu\sum_{J\in\cJ}t_J^2
 +B_d\mu^2\sum_{J=[a,b)\in\cJ}(x_a^2+x_b^2)
 \le B_d\left(\frac{d+1}K+2\right)\mu^2.
\]
Set $C_d:=B_d((d+1)/K+2)$. Combining~\eqref{eq:interpolated-path-energy},
\eqref{eq:sum-norm}, \eqref{eq:sum-endpoint-corrections}, and $g_n=x_n$,
we obtain
\[
 K\cE_P(g)+\frac\gamma2\sum_{J\in\cJ_{\mathrm{str}}}t_J^2
 \le\sum_{J\in\cJ}\bigl(\cE_J(x)-\mu D_J\bigr)
       +\mu c x_n^2+C_d\mu^2
 \le\mu\bigl(\|g\|_2^2+c g_n^2\bigr)+C_d\mu^2.\qedhere
\]
\end{proof}

Combining Proposition~\ref{prop:refined} with~\eqref{eq:coarse-norm},
we obtain the following path comparison. For extremal graphs, the
hypothesis $\mu=O_d(n^{-2})$ will follow from
Lemma~\ref{lem:construction}.
\begin{corollary}[Path comparison]\label{cor:weighted-norm}
For every fixed $C_0>0$, there exist constants
$n_0=n_0(d,C_0)$ and $C=C(d,C_0)>0$ such that, for every
$G\in\mathcal G(n,d)$ with $n\ge n_0$ and $\mu\le C_0n^{-2}$,
the sequence $g$ constructed above satisfies
\[
 \mu\ge
 \frac{K\cE_P(g)+(\gamma/2)\sum_{J\in\cJ_{\mathrm{str}}}t_J^2}
      {\sum_{i=1}^n g_i^2+c g_n^2}
 -\frac{C}{n^4}.
\]
In particular,~\eqref{eq:weighted-comparison} holds whenever
$\mu=O_d(n^{-2})$.
\end{corollary}
\begin{proof}
For $\mu\le C_0n^{-2}$, equation~\eqref{eq:coarse-norm} gives
$\|x-g\|_2\le\sqrt{d(d+1)C_0/K}\,n^{-1}$.
By increasing $n_0=n_0(d,C_0)$ if necessary, we may assume that $\|x-g\|_2\le1/4$ and $\mu\le1/4$
whenever $n\ge n_0$. Since $c>0$ and $\|x\|_2=1$,
\[
 \sum_{i=1}^n g_i^2+c g_n^2
 \ge\|g\|_2^2\ge(1-\|x-g\|_2)^2
 \ge\frac9{16}>\frac12.
\]
Since $\mu\le1/4$, Proposition~\ref{prop:refined} applies.
Dividing~\eqref{eq:refined} by $\sum_{i=1}^n g_i^2+c g_n^2$
gives an error of at most $2C_d\mu^2\le2C_dC_0^2n^{-4}$.
Take $C:=2C_dC_0^2$; if $C_0$ depends only on $d$, this error is
$O_d(n^{-4})$.
\end{proof}
\section{Lower bounds on \texorpdfstring{$d-\lambda_1(G)$}{d - lambda1(G)}}\label{sec:path}
In this section, we use Corollary~\ref{cor:weighted-norm} to derive a
general lower bound on $d-\lambda_1(G)$, together with a strengthened
bound when a strict block begins at index $0$ or $1$. Recall that $d\ge3$ is fixed and odd,
$K=d-1$, and $c=(d-2)/2$; put $\kappa:=K\pi^2/4$.
For $G\in\mathcal G(n,d)$, retain $\mu=d-\lambda_1(G)$ and the sequence
$g$, partition $\cJ$, set $\cJ_{\mathrm{str}}$, and constant $\gamma$
from Section~\ref{sec:cuts}. For $u=(u_0,\ldots,u_n)$ with $u_0=0$,
write $\|u\|_c^2:=\sum_{i=1}^n u_i^2+c u_n^2$.

Fix $C_0>0$. If $\mu\le C_0n^{-2}$, deleting the nonnegative sum over
$\cJ_{\mathrm{str}}$ in Corollary~\ref{cor:weighted-norm} gives
\begin{equation}\label{eq:baseline-from-refined}
 \mu\ge\frac{K\cE_P(g)}{\|g\|_c^2}-\frac{C_1}{n^4},
 \qquad C_1=C_1(d,C_0)>0,
\end{equation}
for all sufficiently large $n$.
Since $g_0=0$, the edge $\{0,1\}$ of $P$ contributes
$K(g_1-g_0)^2=K g_1^2$ to $K\cE_P(g)$.
The following variational formula will be applied to
$(g_1,\ldots,g_n)$ and will also be used in the construction in
Subsection~\ref{sec:construction}.
\begin{lemma}\label{lem:path-minimum}
For fixed $a,k,b>0$, as the integer $q\to\infty$,
\begin{equation}\label{eq:path-endpoint-minimum}
 \min_{z\in\mathbb R^{q+1}\setminus\{0\}}
 \frac{a z_0^2+k\sum_{i=0}^{q-1}(z_{i+1}-z_i)^2}
      {\sum_{i=0}^{q-1}z_i^2+b z_q^2}
 =\frac{k\pi^2}{4(q+b-1/2+k/a)^2}+O_{a,k,b}(q^{-4}).
\end{equation}
Both $z_0$ and $z_q$ are free in this minimum.
\end{lemma}
\begin{proof}
Denote the minimum by $\Lambda$ and normalize
$\sum_{i=0}^{q-1}z_i^2+bz_q^2=1$.
This constraint gives $|z_i|\le1$ for $i<q$ and $|z_q|\le b^{-1/2}$.
The normalized constraint set is compact, and hence the continuous
numerator attains its minimum.
Replacing each $z_i$ by $|z_i|$ preserves the denominator and does not
increase $(z_{i+1}-z_i)^2$, so choose a minimizer with $z_i\ge0$.
The Lagrange multiplier equations at a minimizer give
\begin{equation}\label{eq:path-euler}
 \begin{aligned}
  a z_0+k(z_0-z_1)&=\Lambda z_0,\\
  k(2z_i-z_{i-1}-z_{i+1})&=\Lambda z_i
      &&(1\le i<q),\\
  k(z_q-z_{q-1})&=b\Lambda z_q.
 \end{aligned}
\end{equation}
If $z_i=0$ for $0<i<q$, then $z_{i-1}+z_{i+1}=0$ and hence
$z_{i-1}=z_{i+1}=0$.
The equations at $0,q$ likewise give $z_0=0\Rightarrow z_1=0$ and
$z_q=0\Rightarrow z_{q-1}=0$.
By connectivity along the path, any zero coordinate would therefore
force all coordinates to vanish, contradicting the normalization.
Hence $z_i>0$
for every $i$.
The numerator vanishes only at $z=0$, and hence $\Lambda>0$.
On the other hand, testing the quotient with $z_i=i$ gives
\[
 0<\Lambda\le\frac{kq}{\sum_{i=0}^{q-1}i^2+bq^2}
 =O_{a,k,b}(q^{-2}).
\]
Since $\Lambda=O_{a,k,b}(q^{-2})$, for all sufficiently large $q$
there is a unique $\theta\in(0,\pi)$ such that
\(
 \Lambda=2k(1-\cos\theta).
\)
Moreover,
\(
 \theta=O_{a,k,b}(q^{-1}).
\)
The equations for $1\le i<q$ become
$z_{i+1}-2\cos\theta\,z_i+z_{i-1}=0$.
Thus $z_i=A\sin(i\theta+\beta)$ for $0\le i\le q$, with suitable $A>0$ and $\beta\in(0,\pi)$, where the latter choice is
possible since $z_0>0$.
The first equation of~\eqref{eq:path-euler} gives
\[
 \tan\beta
 =\frac{\sin\theta}{a/k-(1-\cos\theta)}
 =\frac{k}{a}\theta+O_{a,k}(\theta^3).
\]
For small $\theta$, the denominator is positive, so
$0<\beta<\pi/2$ and $\beta=(k/a)\theta+O_{a,k}(\theta^3)$.

We also have $q\theta+\beta<\pi$.
Otherwise, let $j\in\{1,\ldots,q\}$ be the smallest index such that
$j\theta+\beta\ge\pi$. Then
\(
 \pi\le j\theta+\beta<\pi+\theta<2\pi
\),
and hence $z_j\le0$, a contradiction.
Consequently $0<q\theta+\beta<\pi$, and the last equation
of~\eqref{eq:path-euler} yields
\[
 \begin{aligned}
  \cot(q\theta+\beta)&=(2b-1)\tan(\theta/2),\\
  q\theta+\beta
  &=\frac\pi2-\arctan\!\bigl((2b-1)\tan(\theta/2)\bigr)
   =\frac\pi2-\left(b-\frac12\right)\theta+O_b(\theta^3).
 \end{aligned}
\]
Combining this with
\(
 \beta=k\theta/a+O_{a,k}(\theta^3),
\)
we obtain
$\left(q+b-1/2+k/a\right)\theta=\pi/2+O_{a,k,b}(\theta^3)$.
Since $\theta=O_{a,k,b}(q^{-1})$ and
$q+b-1/2+k/a=q+O_{a,k,b}(1)$,
it follows that
$\theta={\pi/[2(q+b-1/2+k/a)]}+O_{a,k,b}(q^{-4})$.
Substitution in $\Lambda=k\theta^2+O_k(\theta^4)$ proves
\eqref{eq:path-endpoint-minimum}.
\end{proof}

We now apply~\eqref{eq:path-endpoint-minimum} to the quotient
in~\eqref{eq:baseline-from-refined}.
\begin{corollary}[The general lower bound]\label{cor:path-baseline}
For every fixed $C_0>0$, there exist constants
$n_0=n_0(d,C_0)$ and $C=C(d,C_0)>0$ 
such that every $G\in\mathcal G(n,d)$ with $n\ge n_0$ and
$\mu\le C_0n^{-2}$ satisfies
\[
 \mu\ge\frac{\kappa}{(n+K/2)^2}-\frac{C}{n^4}.
\]
\end{corollary}
\begin{proof}
The condition $u_0=0$ gives
\[
 \begin{aligned}
  \min_{\substack{u\in\mathbb R^{n+1}\setminus\{0\}\\u_0=0}}
   \frac{K\sum_{i=0}^{n-1}(u_{i+1}-u_i)^2}{\|u\|_c^2}
  &=\min_{(u_1,\ldots,u_n)\ne0}
   \frac{K u_1^2+K\sum_{i=1}^{n-1}(u_{i+1}-u_i)^2}
        {\sum_{i=1}^{n-1}u_i^2+(1+c)u_n^2}\\
  &=\frac{\kappa}{(n+K/2)^2}+O_d(n^{-4}),
 \end{aligned}
\]
where the last equality is~\eqref{eq:path-endpoint-minimum} with
$q=n-1$, $a=k=K$, and $b=1+c$, using $c+1/2=K/2$.
Since $g_0=0$ and $g_n=x_n>0$, the vector
$(g_1,\ldots,g_n)$ is nonzero. Hence the corresponding quotient is at
least the above minimum. Substituting this estimate into~\eqref{eq:baseline-from-refined}
proves the result, after adjusting the constant $C=C(d,C_0)$.
\end{proof}

For $J=[s,s+\ell)\in\cJ_{\mathrm{str}}$, the identity
$g_{i+1}-g_i=t_J/\ell$ for $i\in J$ gives
$(\gamma/2)t_J^2=(\gamma\ell/2)\sum_{i=s}^{s+\ell-1}(g_{i+1}-g_i)^2$.
Thus, retaining this term amounts to replacing the coefficient $K$
of each of these $\ell$ squared differences by
$K+\gamma\ell/2$.
The next lemma evaluates the resulting improvement for $s\in\{0,1\}$,
the locations supplied by Lemma~\ref{lem:initial-strict} in the
degree-sequence proof.
\begin{lemma}[The improvement from a strict block]\label{lem:path-strict}
Put $r_*:=\gamma/(2K+\gamma)>0$.
For every fixed $C_0>0$, there are $n_0=n_0(d,C_0)$ and $C=C(d,C_0)>0$
such that the following holds. If $G\in\mathcal G(n,d)$,
$n\ge n_0$, $\mu\le C_0n^{-2}$, and the partition constructed in
Section~\ref{sec:cuts} contains a strict block $J=[s,s+\ell)$ with
$s\in\{0,1\}$, then
\[
 \mu\ge\frac{\kappa}{(n+K/2-r_*)^2}-\frac{C}{n^4}.
\]
\end{lemma}
\begin{proof}
Fix a strict block $J=[s,s+\ell)$ as in the statement; recall that
$\ell\in\{1,d,d+1\}$. For $0\le i<n$, set
\[
 a_i:=
 \begin{cases}
 K+\gamma\ell/2,&s\le i<s+\ell,\\
 K,&\text{otherwise},
 \end{cases}
 \qquad
 \Lambda:=\min_{\substack{u\in\mathbb R^{n+1}\setminus\{0\}\\u_0=0}}
 \frac{\sum_{i=0}^{n-1}a_i(u_{i+1}-u_i)^2}{\|u\|_c^2}.
\]
Retaining only the term corresponding to $J$ in
Corollary~\ref{cor:weighted-norm}, and using the fact that $g$ is
affine on $J$, we obtain
\begin{equation}\label{eq:strict-from-refined}
 \mu\ge
 \frac{\sum_{i=0}^{n-1}a_i(g_{i+1}-g_i)^2}{\|g\|_c^2}
 -\frac{C_1}{n^4}
 \ge\Lambda-\frac{C_1}{n^4},
 \qquad C_1=C_1(d,C_0)>0.
\end{equation}
Here $g_0=0$ and $g_n=x_n>0$, so $g$ is admissible in the
minimum defining $\Lambda$. It remains to estimate $\Lambda$ to
within $O_d(n^{-4})$.

Put $L:=s+\ell\le d+2$ and assume that $n>L$.
As in Lemma~\ref{lem:path-minimum}, the minimum is attained after
normalizing $\|u\|_c=1$. Replacing the coordinates by their absolute
values preserves the denominator and does not increase the numerator.
We may therefore choose a minimizer with $u_i\ge0$.
The first-order optimality conditions, together with $a_{n-1}=K$,
give
\begin{equation}\label{eq:strict-euler}
 \begin{aligned}
  a_{i-1}(u_i-u_{i-1})-a_i(u_{i+1}-u_i)&=\Lambda u_i
        &&(1\le i<n),\\
  K(u_n-u_{n-1})&=\Lambda(1+c)u_n.
 \end{aligned}
\end{equation}
If $u_i=0$ for some $1\le i<n$, then
$a_{i-1}u_{i-1}+a_i u_{i+1}=0$, so both neighbouring coordinates
vanish. If $u_n=0$, the last equation gives $u_{n-1}=0$.
In either case, repeated application of these equations forces
$u_1=\cdots=u_n=0$, contrary to the normalization.
Thus $u_i>0$ for $1\le i\le n$.

Since $u_0=0$ and all $a_i$ are positive, the numerator vanishes
only at $u=0$. Hence $\Lambda>0$. Testing the quotient with $u_i=i$
also gives
\[
 \Lambda\le
 \frac{Kn+\gamma\ell^2/2}{\sum_{i=1}^n i^2+cn^2}
 =O_d(n^{-2}).
\]
For all sufficiently large $n$, we may therefore write
$\Lambda=2K(1-\cos\theta)$, where $0<\theta<\pi$ and
$\theta=O_d(n^{-1})$.

Since $a_i=K$ for $i\ge L$, we first estimate $u_L$ and $u_{L+1}$
from the equations at $1,\ldots,L$.
As $u_1>0$, we may rescale $u$ so that $a_0u_1=1$.
We use this normalization in place of $\|u\|_c=1$ for the remainder
of the proof; neither $\Lambda$ nor~\eqref{eq:strict-euler} is changed.
Summing the first equation of~\eqref{eq:strict-euler} over
$1,\ldots,i$ gives
$a_i(u_{i+1}-u_i)=1-\Lambda\sum_{j=1}^i u_j$ for $1\le i<n$.
\Needspace{8\baselineskip}

Since $\Lambda>0$ and $u_j>0$, we have
$u_{j+1}-u_j\le a_j^{-1}$ for $1\le j<n$.
Together with $u_0=0$ and $u_1=a_0^{-1}$, this implies
$0<u_i\le\sum_{j=0}^{i-1}1/a_j\le i/K$ for $1\le i\le n$.
Using the formula for $a_j(u_{j+1}-u_j)$ and summing the
increments over $0\le j<i$, we obtain, for $1\le i\le L+1$,
\[
 0\le\sum_{j=0}^{i-1}\frac1{a_j}-u_i
 =\Lambda\sum_{j=1}^{i-1}\frac1{a_j}\sum_{k=1}^{j}u_k
 \le\frac{\Lambda(i-1)i(i+1)}{6K^2}
 =O_d(\Lambda).
\]
The last estimate is uniform because $i\le L+1\le d+3$.
Set $r_\ell:=\gamma\ell^2/(2K+\gamma\ell)$. By the definition of
the coefficients $a_j$,
$K\sum_{j=0}^{L-1}1/a_j=s+{K\ell/(K+\gamma\ell/2)}=L-r_\ell$.
Using $a_L=K$ and $\Lambda=O_d(\theta^2)$, we conclude that
\[
 u_L=\frac{L-r_\ell}{K}+O_d(\theta^2),
 \qquad
 u_{L+1}=\frac{L+1-r_\ell}{K}+O_d(\theta^2).
\]
Moreover,
\[
 L-r_\ell=s+\frac{K\ell}{K+\gamma\ell/2}
 \ge\frac{K}{K+\gamma(d+1)/2}.
\]
Thus, for all sufficiently large $n$,
$u_L\ge(2K+\gamma(d+1))^{-1}>0$, and division gives
$u_{L+1}/u_L=1+{1/(L-r_\ell)}+O_d(\theta^2)$.

For $L+1\le i<n$, equation~\eqref{eq:strict-euler} becomes
\(
 u_{i+1}-2\cos\theta\,u_i+u_{i-1}=0\).
Hence
$u_i=A\sin((i-L)\theta+\beta)$ for $L\le i\le n$,
where $A>0$ and $\beta\in(0,\pi)$; the latter choice is possible
because $u_L>0$. Substituting the estimate for $u_{L+1}/u_L$ gives
\[
 \tan\beta
 =\frac{\sin\theta}{u_{L+1}/u_L-\cos\theta}
 =(L-r_\ell)\theta+O_d(\theta^3).
\]
For sufficiently small $\theta$, the denominator is positive.
Therefore $0<\beta<\pi/2$ and
\(
 \beta=(L-r_\ell)\theta+O_d(\theta^3)
\).
As in Lemma~\ref{lem:path-minimum}, positivity of $u_L,\ldots,u_n$
implies $0<(n-L)\theta+\beta<\pi$.
The last equation of~\eqref{eq:strict-euler} now gives
\[
 \begin{aligned}
 \cot((n-L)\theta+\beta)&=(1+2c)\tan(\theta/2),\\
 (n-L)\theta+\beta
 &=\frac\pi2-\left(c+\frac12\right)\theta+O_d(\theta^3).
 \end{aligned}
\]
Combining this with the expansion for $\beta$ and
$c+1/2=K/2$ gives
$\left(n+K/2-r_\ell\right)\theta=\pi/2+O_d(\theta^3)$.
Since $r_\ell<\ell\le d+1$ and $\theta=O_d(n^{-1})$, we have
$\theta={\pi/[2(n+K/2-r_\ell)]}+O_d(n^{-4})$.
Consequently,
\[
 \Lambda=K\theta^2+O_d(\theta^4)
 =\frac{\kappa}{(n+K/2-r_\ell)^2}+O_d(n^{-4}).
\]
The error estimates are uniform over
$(s,\ell)\in\{0,1\}\times\{1,d,d+1\}$, since this is a finite set.
Also, as $\ell\ge1$,
$r_\ell={\gamma\ell^2/(2K+\gamma\ell)}\ge{\gamma\ell/(2K+\gamma)}\ge r_*$.
For all sufficiently large $n$, the quantities $n+K/2-r_\ell$
and $n+K/2-r_*$ are positive. Hence~\eqref{eq:strict-from-refined}
yields
\[
 \mu\ge\Lambda-\frac{C_1}{n^4}
 \ge\frac{\kappa}{(n+K/2-r_*)^2}
      -O_{d,C_0}(n^{-4}).\qedhere
\]
\end{proof}

\section{Proofs of the main results}\label{sec:finish}
We first construct graphs $G_n\in\mathcal G(n,d)$ that give an upper
bound on $d-\rho(n,d)$. This bound and
Corollary~\ref{cor:path-baseline} agree up to an $O_d(n^{-4})$ error,
proving Theorem~\ref{thm:asymptotic}.
We then show that an extremal graph whose degree sequence is neither
$(d,\ldots,d,d-1)$ nor $(d,\ldots,d,1)$ has a strict block beginning
at $0$ or $1$. Combining the stronger lower bound in
Lemma~\ref{lem:path-strict} with Theorem~\ref{thm:asymptotic}
yields a contradiction, proving
Theorem~\ref{thm:degree-sequence}.
Throughout this section, $d\ge3$ is fixed and odd,
$K=d-1$, and $\kappa=K\pi^2/4$.

\subsection{An upper bound on \texorpdfstring{$d-\rho(n,d)$}{d - rho(n,d)}}\label{sec:construction}
The construction below is based on that of
Liu~\cite[Section~6.1]{Liu} for odd $d$.
We choose the end graph explicitly and use
Lemma~\ref{lem:path-minimum} to obtain an upper bound with an
$O_d(n^{-4})$ error.

\begin{lemma}\label{lem:construction}
For every sufficiently large $n$, there is a graph
$G_n\in\mathcal G(n,d)$ with degree sequence $(d,\ldots,d,d-1)$
when $n$ is odd and $(d,\ldots,d,1)$ when $n$ is even, such that
\[
 d-\lambda_1(G_n)
 \le\frac{\kappa}{(n+K/2)^2}+O_d(n^{-4}).
\]
\end{lemma}
\begin{proof}
Put $m:=d+1$, and let $\varepsilon=0$ when $n$ is odd and
$\varepsilon=1$ when $n$ is even.
The set $\{d+2,d+4,\ldots,2d+1\}$ contains exactly one representative
of each odd residue class modulo $m$. Since $n-\varepsilon$ is odd,
choose $s$ in this set such that $s\equiv n-\varepsilon\pmod m$,
and put $q:=(n-s-\varepsilon)/m$.
Then $n=qm+s+\varepsilon$, and $q\ge2$ for all sufficiently large $n$.
The graph $G_n$ will consist of $q$ copies of $K_m$ with one edge
deleted, an $s$-vertex graph, and, when $\varepsilon=1$, one additional
leaf.

We first construct a simple connected graph $F_s$ of order $s$
with one vertex of degree $d-1$ and all others of degree $d$.
Write $s=2h+1$, let $V(F_s)=\{0,\ldots,s-1\}$, and set
\[
 E(F_s)=
 \bigl\{\{i,i+j\}:0\le i<s,\ 1\le j\le K/2\bigr\}
 \cup\bigl\{\{i,i+h\}:1\le i\le h\bigr\},
\]
where addition is taken modulo $s$.
The first edge set gives each vertex the $K$ distinct neighbours
$i\pm1,\ldots,i\pm K/2$ and contains a cycle through all $s$ vertices.
The second edge set is a matching that pairs
$\{1,\ldots,h\}$ with $\{h+1,\ldots,2h\}$.
These matching edges are not in the first edge set, since
$h,h+1>K/2$. Thus $F_s$ is simple and connected, with
$d_{F_s}(0)=d-1$ and $d_{F_s}(i)=d$ for $1\le i<s$.

Take pairwise disjoint copies $B_0,\ldots,B_{q-1}$ of $K_m$ with
one edge deleted, also disjoint from $F_s$.
For $0\le j<q$, let $u_j,b_j$ be the endpoints of the deleted
edge in $B_j$. Denote vertex $0$ of $F_s$ by $u_q$, and add
the edges $b_ju_{j+1}$ for $0\le j<q$.
These edges join $B_0,\ldots,B_{q-1},F_s$ in this order.
Each $b_j$ and $u_{j+1}$ acquires one edge and has degree $d$;
all other vertices except $u_0$ also have degree $d$.
If $\varepsilon=1$, add a new vertex $p$ and the edge $pu_0$.
The resulting graph $G_n$ is simple and connected and has order
$qm+s+\varepsilon=n$.
When $\varepsilon=0$, only $u_0$ has degree below $d$, and its degree
is $d-1$. When $\varepsilon=1$, only $p$ has degree below $d$,
and its degree is $1$.
Thus $G_n\in\mathcal G(n,d)$ has the required degree sequence.
\Needspace{8\baselineskip}

Since the smallest eigenvalue of $dI-A(G_n)$ is
$d-\lambda_1(G_n)$, the Rayleigh principle gives
\[
 d-\lambda_1(G_n)\le
 \frac{f^{\trans}(dI-A(G_n))f}{\|f\|_2^2}
\]
for every nonzero $f:V(G_n)\to\mathbb R$.
We construct such an $f$ from a nonzero sequence
$z=(z_0,\ldots,z_q)$.
Put $t_j:=z_{j+1}-z_j$ for $0\le j<q$, and prescribe
$f(u_j)=z_j$ for $0\le j\le q$.

For each $0\le j<q$, minimize
\[
 \sum_{vw\in E(B_j)\cup\{b_ju_{j+1}\}}(f(v)-f(w))^2
\]
over the values on $V(B_j)\setminus\{u_j\}$, with
$f(u_j)=z_j$ and $f(u_{j+1})=z_{j+1}$ fixed.
The graph on $V(B_j)\cup\{u_{j+1}\}$ with this edge set is connected,
so the minimizer is unique by Lemma~\ref{lem:prescribed-values}.
Every permutation of the $K$ vertices in
$V(B_j)\setminus\{u_j,b_j\}$ preserves this graph and fixes the two
boundary vertices. By uniqueness, $f$ has the same value at these
$K$ vertices.
The optimality equations at any such vertex $v$ and at $b_j$ are
$2f(v)=z_j+f(b_j)$ and $df(b_j)=Kf(v)+z_{j+1}$.
Solving these equations gives
\[
 f(v)=z_j+\frac{t_j}{m}
 \quad\bigl(v\in V(B_j)\setminus\{u_j,b_j\}\bigr),
 \qquad
 f(b_j)=z_j+\frac{2t_j}{m}.
\]
Set $f(v)=z_q$ for every $v\in V(F_s)$.
This agrees with the prescribed value at $u_q$.

When $\varepsilon=1$, the edge $pu_0$ and the degree deficiency
at $p$ contribute
\[
 (z_0-f(p))^2+Kf(p)^2
 =d\left(f(p)-\frac{z_0}{d}\right)^2+\frac Kd z_0^2
\]
to $f^{\trans}(dI-A(G_n))f$.
We therefore set $f(p)=z_0/d$.
The function $f$ is now defined on all of $V(G_n)$ and is nonzero,
since $f(u_j)=z_j$ for $0\le j\le q$.

For each $0\le j<q$, the energy on $E(B_j)\cup\{b_ju_{j+1}\}$ is
\[
 \sum_{vw\in E(B_j)\cup\{b_ju_{j+1}\}}(f(v)-f(w))^2
 =2K\left(\frac{t_j}{m}\right)^2
  +\left(\frac{Kt_j}{m}\right)^2
 =\frac Km t_j^2.
\]
The edges of $F_s$ contribute zero.
When $\varepsilon=0$, the only degree-deficiency term is $z_0^2$
at $u_0$. When $\varepsilon=1$, the edge $pu_0$ and the
degree-deficiency term at $p$ together contribute $(K/d)z_0^2$.
Expanding $f^{\trans}(dI-A(G_n))f$ as
in~\eqref{eq:energy-identity}, we obtain
\begin{equation}\label{eq:construction-energy}
 f^{\trans}(dI-A(G_n))f
 =\left(1-\frac{\varepsilon}{d}\right)z_0^2
  +\frac Km\sum_{j=0}^{q-1}t_j^2.
\end{equation}

We next compute $\|f\|_2^2$.
For each $0\le j<q$,
\[
 \begin{aligned}
 \sum_{v\in V(B_j)}f(v)^2
 &=z_j^2+K\left(z_j+\frac{t_j}{m}\right)^2
       +\left(z_j+\frac{2t_j}{m}\right)^2\\
 &=mz_j^2+(z_{j+1}^2-z_j^2)
       -\frac{K(d+2)}{m^2}t_j^2.
 \end{aligned}
\]
Summing over $j$, the differences $z_{j+1}^2-z_j^2$ telescope
to $z_q^2-z_0^2$.
Adding the contribution $s z_q^2$ from $F_s$ and, when present,
the contribution $z_0^2/d^2$ from the leaf gives
\begin{equation}\label{eq:construction-norm}
 \|f\|_2^2
 =m\sum_{j=0}^{q-1}z_j^2+(s+1)z_q^2
  -\left(1-\frac{\varepsilon}{d^2}\right)z_0^2
  -\frac{K(d+2)}{m^2}\sum_{j=0}^{q-1}t_j^2.
\end{equation}

To apply Lemma~\ref{lem:path-minimum}, define
\begin{equation}\label{eq:construction-minimum}
 \Lambda:=\min_{z\in\mathbb R^{q+1}\setminus\{0\}}
 \frac{(1-\varepsilon/d)z_0^2
       +(K/m)\sum_{j=0}^{q-1}(z_{j+1}-z_j)^2}
      {m\sum_{j=0}^{q-1}z_j^2+(s+1)z_q^2}.
\end{equation}
By~\eqref{eq:construction-norm}, the denominator in this auxiliary
minimum exceeds $\|f\|_2^2$ by
\[
 \left(1-\frac{\varepsilon}{d^2}\right)z_0^2
 +\frac{K(d+2)}{m^2}\sum_{j=0}^{q-1}t_j^2.
\]
We first evaluate $\Lambda$ and then control this difference for
a minimizing vector.

Dividing the numerator and denominator
in~\eqref{eq:construction-minimum} by $m$ gives the minimum
in Lemma~\ref{lem:path-minimum} with
$a={(1-\varepsilon/d)/m}$, $k=K/m^2$, and $b={(s+1)/m}$.
All three parameters are positive.
Since $\varepsilon\in\{0,1\}$, we have
$K/(1-\varepsilon/d)=K+\varepsilon$, and hence
\[
 q+b-\frac12+\frac ka
 =q+\frac{s+1}{m}-\frac12+\frac{K+\varepsilon}{m}
 =\frac{n+K/2}{m}.
\]
It follows from~\eqref{eq:path-endpoint-minimum} that
$\Lambda={\kappa/(n+K/2)^2}+O_d(n^{-4})$.
Here $q=n/m+O_d(1)$, and the possible values of $s$ and
$\varepsilon$ form a finite set depending only on $d$.
Thus the error constant depends only on $d$.

Choose a minimizer $z$ in~\eqref{eq:construction-minimum},
normalized so that
$m\sum_{j=0}^{q-1}z_j^2+(s+1)z_q^2=1$,
and let $f$ be the corresponding function on $G_n$.
By~\eqref{eq:construction-energy},
$f^{\trans}(dI-A(G_n))f=\Lambda$.
The coefficient inequalities
\[
 1-\frac{\varepsilon}{d^2}
 \le\frac md\left(1-\frac{\varepsilon}{d}\right),
 \qquad
 \frac{K(d+2)}{m^2}\le\frac Kd
\]
and~\eqref{eq:construction-norm} give
\[
 0\le1-\|f\|_2^2
 =\left(1-\frac{\varepsilon}{d^2}\right)z_0^2
  +\frac{K(d+2)}{m^2}\sum_{j=0}^{q-1}t_j^2
 \le\frac md\Lambda.
\]
Since $\Lambda=O_d(n^{-2})$, for all sufficiently large $n$,
$\|f\|_2^2\ge1-(m/d)\Lambda\ge1/2$.
Applying the Rayleigh principle to $f$, we conclude that
\begin{align*}
 d-\lambda_1(G_n)
 &\le\frac{f^{\trans}(dI-A(G_n))f}{\|f\|_2^2}
 \le\frac{\Lambda}{1-(m/d)\Lambda}\\
 &=\Lambda+O_d(\Lambda^2)
 =\frac{\kappa}{(n+K/2)^2}+O_d(n^{-4}).\qedhere
\end{align*}
\end{proof}

\begin{proof}[\textbf{Proof of Theorem~\ref{thm:asymptotic}}]
Let $G\in\mathcal G(n,d)$ be extremal, and let $G_n$ be the graph
in Lemma~\ref{lem:construction}. Put
$\mu:=d-\lambda_1(G)=d-\rho(n,d)$.
By extremality and Lemma~\ref{lem:construction},
\[
 \mu\le d-\lambda_1(G_n)
 \le\frac{\kappa}{(n+K/2)^2}+O_d(n^{-4}).
\]
Thus $\mu\le C_0n^{-2}$ for some $C_0=C_0(d)>0$ and all
sufficiently large $n$.
Applying Corollary~\ref{cor:path-baseline} gives
$\mu\ge{\kappa/(n+K/2)^2}-O_d(n^{-4})$.
Combining the two bounds, we obtain
$\mu={\kappa/(n+K/2)^2}+O_d(n^{-4})$.
Finally,
${1/(n+K/2)^2}=1/n^2-K/n^3+O_d(n^{-4})$.
Substituting $\rho(n,d)=d-\mu$, $K=d-1$, and
$\kappa=K\pi^2/4$ proves~\eqref{eq:main}.
\end{proof}

\subsection{Degree sequences of extremal graphs}
The next lemma identifies the graphs to which
Lemma~\ref{lem:path-strict} applies.
It shows that any $G\in\mathcal G(n,d)$ whose degree sequence is
neither $(d,\ldots,d,d-1)$ nor $(d,\ldots,d,1)$ has a strict block
beginning at $0$ or $1$.

\begin{lemma}\label{lem:initial-strict}
Let $G\in\mathcal G(n,d)$, order its vertices as in
Section~\ref{sec:cuts}, and let $\cJ$ be the partition constructed
in the proof of Lemma~\ref{lem:block-partition}.
If $\pi(G)$ is neither $(d,\ldots,d,d-1)$ nor $(d,\ldots,d,1)$,
then $\cJ$ contains a strict block $J=[s,s+\ell)$ with
$s\in\{0,1\}$ and $\ell\in\{1,d,d+1\}$.
\end{lemma}
\begin{proof}
Recall that
\[
 q_0=d_{G^+}(v_0)=\sum_{v\in V(G)}\delta_G(v)\ge1,
 \qquad
 0\le\delta_G(v)=d-d_G(v)\le K.
\]
The upper bound follows from the connectedness of $G$.
Since the deficiencies are nonnegative integers,
$\pi(G)=(d,\ldots,d,d-1)$ if and only if $q_0=1$.
Similarly, $\pi(G)=(d,\ldots,d,1)$ if and only if $q_0=K$ and
exactly one vertex has positive deficiency.
We consider the remaining possibilities.

If $q_0>K$, then $0$ is not an index with $q_i<K$, and no interval
beginning at a positive index can contain $0$.
Thus the construction in Lemma~\ref{lem:block-partition} gives
the singleton block $J=\{0\}$, for which $C_J=q_0>K$.

If $2\le q_0<K$, the block $J$ beginning at $0$ has length
$d$ or $d+1$.
Lemma~\ref{lem:block-energy} gives $C_J>K/\ell_J$ in the
length-$d$ case, and allows equality in the length-$(d+1)$ case
only when $q_0=1$. Thus $J$ is strict.

It remains to consider $q_0=K$ with at least two positive deficiencies.
As above, $\{0\}$ is a singleton block, now with $C_{\{0\}}=K$.
Consequently, the next block begins at $1$.
Since $S_1=\{v_0,v_1\}$, the cut $\partial S_1$ consists of
$q_0-\delta_G(v_1)$ edges incident with $v_0$ and
$d_G(v_1)$ edges incident with $v_1$. Therefore
$q_1=q_0-\delta_G(v_1)+d_G(v_1)=2d-1-2\delta_G(v_1)$.
As the sum of the deficiencies is $K$ and at least two are positive,
$\delta_G(v_1)\le K-1=d-2$. Hence $q_1\ge3$.
Moreover, $q_1$ is odd whereas $K$ is even, so $q_1\ne K$.
If $q_1>K$, the partition rule gives the singleton block $J=\{1\}$,
which is strict.
If $q_1<K$, the block $J$ beginning at $1$ has length $d$ or $d+1$.
Since $q_1\ge3$, Lemma~\ref{lem:block-energy} again gives
$C_J>K/\ell_J$. This completes the proof.
\end{proof}

\begin{proof}[\textbf{Proof of Theorem~\ref{thm:degree-sequence}}]
Let $G\in\mathcal G(n,d)$ be extremal.
By Theorem~\ref{thm:asymptotic},
\[
 \mu:=d-\lambda_1(G)
 =\frac{\kappa}{(n+K/2)^2}+O_d(n^{-4}).
\]
Suppose, to the contrary, that $\pi(G)$ is neither
$(d,\ldots,d,d-1)$ nor $(d,\ldots,d,1)$.
By Lemma~\ref{lem:initial-strict}, the partition $\cJ$ contains
a strict block beginning at $0$ or $1$.
Since $\mu=O_d(n^{-2})$, Lemma~\ref{lem:path-strict} applies and gives
\[
 \mu\ge\frac{\kappa}{(n+K/2-r_*)^2}-O_d(n^{-4}),
 \qquad r_*=\frac{\gamma}{2K+\gamma}>0.
\]
For fixed $d$,
\[
 \frac{\kappa}{(n+K/2-r_*)^2}
 =\frac{\kappa}{(n+K/2)^2}
  +\frac{2\kappa r_*}{n^3}+O_d(n^{-4}).
\]
Comparing these estimates with Theorem~\ref{thm:asymptotic}, we obtain
$2\kappa r_*/n^3\le C/n^4$
for some $C=C(d)>0$.
This is impossible for all sufficiently large $n$, since
$\kappa r_*>0$.
Therefore $\pi(G)$ is $(d,\ldots,d,d-1)$ or $(d,\ldots,d,1)$.

It remains to determine which sequence is possible for each parity
of $n$. Since $d$ is odd,
$\sum_{v\in V(G)}\delta_G(v)=nd-2|E(G)|\equiv n\pmod2$.
The total deficiency is $1$ for $(d,\ldots,d,d-1)$ and
$K=d-1$ for $(d,\ldots,d,1)$.
The former is odd and the latter is even, proving~\eqref{eq:degrees}.
\end{proof}

\appendix
\Needspace{10\baselineskip}
\section{Local estimates and a spectral lower bound}\label{app:ordered-cut}
We prove the coordinate estimate used in Lemma~\ref{lem:block-energy}
and derive an explicit lower bound on $d-\lambda_1(G)$ for every
$G\in\mathcal G(n,d)$.
Throughout this appendix, $d\ge3$ is odd and $K=d-1$.
Parallel edges are counted with multiplicity.

\subsection{A bound on the minimizing coordinates}\label{app:local-energy}
In the case $ab\notin E(H_J)$ considered in the proof of
Lemma~\ref{lem:block-energy}, the minimizing function with
$w_a=1$ and $w_b=0$ satisfies
$C_J\ge q_a\left(1-\max_{a<i<b}w_i\right)$.
The following lemma supplies the required coordinate bound by taking
$H=H_J$, $(L,R)=(a,b)$, $m=\ell_J-1$, and $k=q_a$.
\Needspace{9\baselineskip}

\begin{lemma}\label{lem:local-energy}
Let $H$ be a connected loopless multigraph, and let $L,R$ be
distinct nonadjacent vertices.
Suppose that $F:=H-\{L,R\}$ is simple, with
$m\in\{d-1,d\}$ vertices, that $d_H(i)=d$ for every $i\in V(F)$,
and that $k:=d_H(L)$ satisfies $1\le k\le d-2$.
The graph $F$ need not be connected.
If $w$ minimizes $\cE_H(w)$ subject to $w_L=1$ and $w_R=0$, then
\[
 \max_{i\in V(F)}w_i
 \le\frac{(d-m+2)k}{(d-m+1)d+k}.
\]
\end{lemma}
\begin{proof}
For $i\in V(F)$, let $p_i$ be the number of edges joining $L$ to $i$,
and let $N_F(i)$ be its set of neighbours in $F$.
Since $LR\notin E(H)$,
$\sum_{i\in V(F)}p_i=k$ and $d_F(i)+p_i\le d$ for $i\in V(F)$.
By Lemma~\ref{lem:prescribed-values}, $0\le w_i\le1$ for all
$i\in V(H)$.
Differentiating the energy with respect to each interior coordinate
gives
\begin{equation}\label{eq:minimizer-local}
 dw_i=p_i+\sum_{j\in N_F(i)}w_j
 \qquad(i\in V(F)).
\end{equation}
Choose $r\in V(F)$ such that $w_r=\max_{i\in V(F)}w_i$, and set
$S:=\sum_{i\in V(F)}w_i$ and $\beta:=d-m+2\in\{2,3\}$.
Since $F$ is simple and $w_i\ge0$,
$\sum_{j\in N_F(i)}w_j\le S-w_i$.
Thus~\eqref{eq:minimizer-local} implies
$(d+1)w_i\le p_i+S$ for every $i\in V(F)$.
Summing these inequalities over $V(F)\setminus\{r\}$ gives
$(d+1)(S-w_r)\le k-p_r+(m-1)S$,
and hence
$\beta S\le(d+1)w_r+k-p_r$.

Summing $(d+1)w_j\le p_j+S$ over $j\in N_F(r)$, we also obtain
\[
 (d+1)\sum_{j\in N_F(r)}w_j
 \le\sum_{j\in N_F(r)}p_j+d_F(r)S
 \le k-p_r+(d-p_r)S.
\]
Here we used $\sum_{j\in N_F(r)}p_j\le k-p_r$,
$d_F(r)\le d-p_r$, and $S\ge0$.
Applying~\eqref{eq:minimizer-local} at $r$ yields
$d(d+1)w_r\le d p_r+k+(d-p_r)S$.
Since $d-p_r\ge2$, we may substitute
$S\le((d+1)w_r+k-p_r)/\beta$ into the right-hand side.
Rearranging, and noting that $(\beta-1)d+p_r>0$, gives
\begin{equation}\label{eq:max-local}
 \begin{aligned}
 w_r
 &\le
 \frac{\beta(d p_r+k)+(d-p_r)(k-p_r)}
      {(d+1)((\beta-1)d+p_r)}\\
 &=\frac{p_r-k}{d+1}
   +\frac{\beta k}{(\beta-1)d+p_r}.
 \end{aligned}
\end{equation}
To compare the last expression with $\beta k/((\beta-1)d+k)$,
observe that $0\le p_r\le k\le d-2$ implies
\[
 \begin{aligned}
 &\bigl((\beta-1)d+k\bigr)\bigl((\beta-1)d+p_r\bigr)
       -\beta k(d+1)\\
 &\quad\ge
 \begin{cases}
 d(d-k)-2k\ge4,&\beta=2,\\
 4d^2-k(d+3)\ge3d^2-d+6>0,&\beta=3.
 \end{cases}
 \end{aligned}
\]
Consequently,
\[
 \begin{aligned}
 \frac{\beta k}{(\beta-1)d+p_r}
 -\frac{\beta k}{(\beta-1)d+k}
 &=
 \frac{\beta k(k-p_r)}
 {\bigl((\beta-1)d+p_r\bigr)\bigl((\beta-1)d+k\bigr)}\\
 &\le\frac{k-p_r}{d+1}.
 \end{aligned}
\]
Substitution in~\eqref{eq:max-local} gives
$w_r\le{\beta k/((\beta-1)d+k)}$.
The result follows from $\beta=d-m+2$ and the choice of $r$.
\end{proof}

\Needspace{12\baselineskip}
\subsection{An explicit spectral lower bound}
The estimates~\eqref{eq:coarse-energy} and~\eqref{eq:coarse-norm}
do not require the assumption $d-\lambda_1(G)=O_d(n^{-2})$.
We combine them with the exact minimum of
$\cE_P(z)/\|z\|_2^2$ under the condition $z_0=0$.

\begin{corollary}\label{cor:explicit}
Let $G\in\mathcal G(n,d)$, where $d\ge3$ is odd, and put
$\nu_n:=2-2\cos(\pi/(2n+1))$. Then
\begin{equation}\label{eq:explicit}
 d-\lambda_1(G)\ge
 \frac{K\nu_n}{\bigl(1+\sqrt{d(d+1)\nu_n}\bigr)^2}.
\end{equation}
In particular,
$d-\lambda_1(G)\ge K\pi^2/(4n^2)-O_d(n^{-3})$.
\end{corollary}
\begin{proof}
Put $v_i:=\sin(i\pi/(2n+1))$ for $0\le i\le n+1$.
Then $v_0=0$, $v_{n+1}=v_n$, and $v_i>0$ for $1\le i\le n$.
The sine addition formula gives
\[
 2v_i-v_{i-1}-v_{i+1}=\nu_n v_i
 \quad(1\le i<n),
 \qquad
 v_n-v_{n-1}=\nu_n v_n.
\]
For every real sequence $z=(z_0,\ldots,z_n)$ with $z_0=0$,
expanding the energy and using these identities yields
\[
 \begin{aligned}
 \cE_P(z)-\nu_n\|z\|_2^2
 &=\sum_{i=1}^{n-1}(2-\nu_n)z_i^2
   +(1-\nu_n)z_n^2
   -2\sum_{i=1}^{n-1}z_i z_{i+1}\\
 &=\sum_{i=1}^{n-1}v_i v_{i+1}
   \left(\frac{z_{i+1}}{v_{i+1}}-\frac{z_i}{v_i}\right)^2
 \ge0.
 \end{aligned}
\]
Equality is attained at $z=(v_0,\ldots,v_n)$.
Therefore
\[
 \nu_n=
 \min_{\substack{z\in\mathbb R^{n+1}\setminus\{0\}\\z_0=0}}
 \frac{\cE_P(z)}{\|z\|_2^2}.
\]

Let $x$ be the ordered unit Perron vector of $G$, extended by $x_0=0$,
and let $g$ be the sequence constructed in Section~\ref{sec:cuts}.
Write $\mu:=d-\lambda_1(G)$.
Applying the preceding inequality to $g$ and
using~\eqref{eq:coarse-energy}, we obtain
$\nu_n\|g\|_2^2\le\cE_P(g)\le\mu/K$.
Also,~\eqref{eq:coarse-norm} gives
$\|x-g\|_2\le\sqrt{d(d+1)\mu/K}$.
The triangle inequality now implies
\[
 1=\|x\|_2
 \le\|g\|_2+\|x-g\|_2
 \le\sqrt{\frac{\mu}{K\nu_n}}
    +\sqrt{\frac{d(d+1)}K\,\mu}.
\]
Rearranging gives
\[
 \sqrt{\mu}\ge
 \frac{\sqrt{K\nu_n}}{1+\sqrt{d(d+1)\nu_n}}.
\]
Squaring proves~\eqref{eq:explicit}.
Finally, the expansions
\[
 \nu_n=\frac{\pi^2}{4n^2}+O(n^{-3}),
 \qquad
 \bigl(1+\sqrt{d(d+1)\nu_n}\bigr)^{-2}
 =1+O_d(n^{-1})
\]
give the stated asymptotic bound.
\end{proof}

\end{document}